\documentclass[11pt]{article}
\usepackage{amsmath,amssymb}
\usepackage{bm}
\usepackage{graphicx}
\usepackage{float}
\usepackage[T1]{fontenc}
\usepackage{textcomp}
\usepackage{type1cm}
\usepackage{pifont}
\usepackage{amsthm}
\usepackage{amscd}
\usepackage{ascmac}
\usepackage{mathrsfs}
\usepackage{alltt}
\usepackage{braket}
\usepackage{cases}
\usepackage[]{multicol}
\makeatletter
\@addtoreset{equation}{section}

\makeatother
\newtheorem{dfn}{Definition}[section]
\newtheorem{thm}{Theorem}[section]
\newtheorem{cor}{Corollary}[section]

\newtheorem{ex}{Example}[section]
\newtheorem{prp}{Proposition}[section]
\newtheorem{rem}{Remark}[section]
\newtheorem{lem}{Lemma}[section]
\newcommand\Res{\mathop{\mathrm{Res}}\nolimits}

\DeclareMathOperator{\Z}{\mathbb{Z}}

\DeclareMathOperator{\R}{\mathbb{R}}
\DeclareMathOperator{\C}{\mathbb{C}}
\DeclareMathOperator{\N}{\mathbb{N}}

\usepackage{ascmac} 
\usepackage{float}
\usepackage{tabularx}

\title{Braid zeta function and some formulae for the torus type\footnote{$\mathbf{Keywords}$ : braid group, dynamical zeta function, knot theory.}}
\author{Kentaro Okamoto\footnote{$\mathbf{2010}$ $\mathbf{Mathematics}$ $\mathbf{Subject}$ $\mathbf{Classification}$ :  20F36,  37C30 }}
\date{\empty}
\begin{document}
\maketitle
\begin{abstract}

There is a well-known zeta function of the $\Z$-dynamical system generated by an element of the symmetric group. By considering this zeta function as a model, 
we can construct a new zeta function of an element of the braid group. 
 In this paper, we show that the Alexander polynomial which is the most classical polynomial invariant of knots can be expressed in terms of this braid zeta function.
\hspace{0.2em}Furthermore we define the function $Z_q$ associated with some braids. We show that this function can be expressed by some braid zeta function for the case of special braids whose closures are isotopic to certain torus knots.
\end{abstract}
\section{Introduction}

\hspace{1.5em}Let $\mathrm{S}_n$ be the symmetric group acting on the finite set $X_n:=\{1,2,\ldots,n\}$. Then, for any permutation $\sigma \in \mathrm{S}_n$,  the zeta function of $\sigma$ is defined as
\begin{align*}
\zeta_{\sigma}(s):=\exp\biggl\{\sum_{m=1}^{\infty}\frac{\# \mathrm{Fix}(\sigma^m,X_n)}{m}s^m\biggr\},
\end{align*}
where $\mathrm{Fix}(\sigma^m,X_n):=\{x\in X_n \mid \sigma^m x=x \}$.
In \cite[Proposition.1]{siryo7}, the following interesting proposition is shown by Kim, Koyama and Kurokawa.
\begin{thm}
For any permutation $\sigma \in \mathrm{S}_n$, the zeta function $\zeta_{\sigma}(s)$ has the following properties.\\
$(1)$ Let $\mathrm{Cycle(\sigma)}$ be the set of primitive cycles of $\sigma$, and $l(P)$ be the length of cycle $P\in \mathrm{Cycle(\sigma)}$. Then, $\zeta_{\sigma}(s)$ has the Euler product:
\begin{align*}
\zeta_{\sigma}(s) = \prod_{P\in \mathrm{Cycle}(\sigma)}\frac{1}{1-s^{l(P)}}.
\end{align*}
$(2)$ Let $p_n : \mathrm{S}_n \longrightarrow \mathrm{GL}_n (\Z)$ be the permutation representation. Then, $\zeta_{\sigma}(s)$ has the following determinant expression:
\begin{align*}
\zeta_{\sigma}(s) = \mathrm{det}(I_n - p_n (\sigma)s)^{-1}.
\end{align*}
$(3)$ $\zeta_{\sigma}(s)$ satisfies the following functional equation:
\begin{align*}
\zeta_{\sigma}(s) = \mathrm{sgn}(\sigma)(-s)^{-n}\zeta_{\sigma}(1/s),
\end{align*}
where $\mathrm{sgn}:\mathrm{S}_n \longrightarrow \set{\pm1}$ is the signature of the permutation.\\
$(4)$ $\zeta_{\sigma}(e^{-s})$ satisfies an analogue of the Riemann hypothesis: all poles of $\zeta_{\sigma}(e^{-s})$ satisfy
\begin{align*}
\mathrm{Re}(s)=0.
\end{align*}
\end{thm}

We call a permutation $\sigma \in \mathrm{S}_n$ {\it simple cycle} if $\sigma$ is a primitive cycle in itself.
Then, for a simple cycle $\sigma \in \mathrm{S}_n$, the zeta function $\zeta_{\sigma}(s)$ has a simple pole at $s=1$ and we obtain
\begin{equation}
\underset{s=1}{\Res} \hspace{0.2em}\zeta_{\sigma}(s) = \lim_{s \to 1}(s-1)(1-s^n)^{-1}=-\frac{1}{n}.
\end{equation}
Remark that the residue of $\zeta_{\sigma}(s)$ at $s=1$ gives us only the information of the length of $\sigma$.

Our first goal is to generalize such properties to the case of the braid group. Consequently, we generalize $\zeta_{\sigma}(s)$ to the zeta function of a braid by using the Burau representation of the braid group. As an application the Alexander polynomial $\Delta_{K}(q)$ which is the most classical invariant of knots can be expressed by the residue of this new zeta function. This is analogous to the fact that the residue of the Dedekind  zeta function at $s=1$ has invariants of an algebraic field such as the class number, discriminant and regulator. Noting that the Burau representation is $q$-deformation of the permutation representation, we obtain a fact that $\Delta_{K}(1)=1$.

We first recall the notations and settings on the braid group briefly. We refer to \cite{siryo3}, \cite{siryo6} and \cite{siryo12} for more details. 

Let us denote the braid group on $n$ strands by $\mathrm{B}_n$. It is known that $\mathrm{B}_n$ has the following presentation:
\begin{align*}
\mathrm{B}_n:=\langle \sigma_i \ (1\leq i \leq n-1) \mid \sigma _i\sigma _j=\sigma _j\sigma _i \ (|i-j|\geq 2), \sigma _i\sigma _{i+1}\sigma _i=\sigma _{i+1}\sigma _i\sigma _{i+1} \ (1\leq i \leq n-2) \rangle.
\end{align*} 
The generator $\sigma_i$ can be identified with the crossing between the $i$-th and $(i+1)$-st strands as Figure $1$ (see \cite[Theorem1.8]{siryo3}) , 
\begin{figure}[!h]
\begin{center}
\begin{tabular}{c}
\begin{minipage}{0.50\hsize}
\begin{center}
\includegraphics[,width=4.3cm]{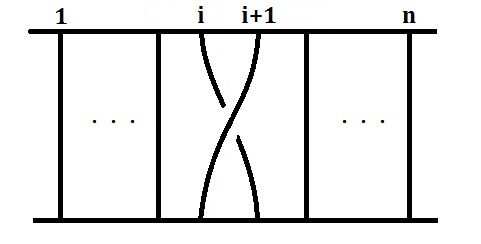} 
\caption{generator $\sigma_i$}
\end{center}
\end{minipage}
\begin{minipage}{0.50\hsize}
\begin{center}
\includegraphics[,width=5.0cm]{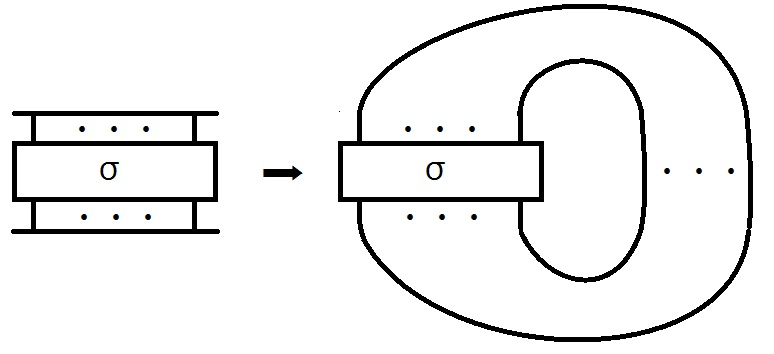} 
\caption{the closure of a braid $\sigma$}
\end{center}
\end{minipage}
\end{tabular}
\end{center}
\end{figure}
and the multiplication of generators implies that the braid obtained by attaching the generators from the top to the bottom.
The ${\it closure}$ of a braid is the link obtained from the braid by connecting upper ends and lower ends as Figure $2$. The closure of $\sigma$ is denoted by $\hat{\sigma}$. 
Let $\beta_{n,q}$ be the ${\it Burau} \hspace{0.3em} {\it representation}$, which is defined by 
\begin{equation}
\label{01}
\beta_{n,q}(\sigma_i) := I_{i-1}\oplus \left(
\begin{array}{ccc}
1-q & 1 \\
q & 0 \\
\end{array}\right) \oplus I_{n-i-1} \in \mathrm{GL}_n (\Lambda).
\end{equation}
Here $\Lambda:= \Z[q^{\pm 1}]$. We also define the  braid zeta function of $\sigma \in \mathrm{B}_n$ by the determinant expression:
\begin{equation}
\zeta(s,\sigma ; \beta_{n,q})=\mathrm{det}(I_n-\beta_{n,q}(\sigma)s)^{-1}.
\end{equation}
We now can state the main result of this paper.
\begin{thm}
$(1)$ For $\sigma \in \mathrm{B}_n$, $\zeta(s,\sigma ;\beta_{n,q})$ satisfies the functional equation:
\begin{align*}
\zeta(s,\sigma ;\beta_{n,q})=\mathrm{sgn}_q(\sigma^{-1})(-s)^{-n}\zeta(1/s,\sigma^{-1} ;\beta_{n,q}).
\end{align*}
Here $\mathrm{sgn}_q(\sigma):=\mathrm{det}(\beta_{n,q}(\sigma))$.\\
$(2)$ If $\hat{\sigma}$ is a knot, the residue of $\zeta(s,\sigma ;\beta_{n,q})$ at $s=1$ is given as follows:
\begin{align*}
\underset{s=1}{\Res}\hspace{0.2em} \zeta(s,\sigma; \beta_{n,q}) = -\frac{1}{[n]_q}\Delta_{\hat{\sigma}}(q)^{-1}.
\end{align*}
Here $\Delta_{\hat{\sigma}}(q)$ is the Alexander polynomial of a knot  $\hat{\sigma}$, and $[n]_q$ is the $q$-integer defined by
\begin{align*}
[n]_q := \frac{1-q^n}{1-q}.
\end{align*}
$(3)$ Assume that $q$ is a point of the unit circle on the complex plane, in other words, $q$ is expressed by $e^{i\theta}(\theta \in \R) $, and that the argument of $q$ satisfies $|\theta|<{2\pi}/n$. Then for any $\sigma \in \mathrm{B}_n$, the braid zeta function of $\sigma$ satisfies an analogue of Riemann hypothesis: all poles of $\zeta(e^{-s},\sigma ;\beta_{n,q})$ satisfy
\begin{align*}
\mathrm{Re}(s)=0.
\end{align*}
\end{thm}

Remark that $\zeta(s,\sigma; \beta_{n,q})$ does not have the Euler product expression, however, $(1)$ and $(3)$ are analogous to Theorem $1.1$. Furthermore, $(2)$ is the characteristic property of the braid zeta function. We prove Theorem $1.2$ in the next section. As corollaries of Theorem $1.2$, we obtain the generating function expression of $\zeta(s, \sigma; \beta_{n,q})$ for $\sigma \in \mathrm{B}_n$. 

Moreover, in the last section, for two braids $b, b'$ we introduce the function $Z_q(s, b,b')$ which is a $q$-analogue of $\zeta(s, \sigma \otimes \tau)$ for two permutations $\sigma, \tau$ (see \cite{siryo7}). Then we prove some formulae of the residue of the function $Z_q$ at $s=1$ for the special braid whose closure is isotopic to the torus knot.

\section{The proof of Theorem\ 1.2}
Before the proof of Theorem $1.2$, we introduce the notion of representation zeta function.
\begin{dfn}\normalfont
\label{def2.1}
Let $G$ be a group and a pair $(\rho,V)$ be a finite-dimensional representation of $G$. Then, for an element $g \in G$, we define the zeta function as
\begin{align*}
\zeta(s, g ;\rho):=\mathrm{det}(I_n - \rho (g)s)^{-1},
\end{align*}
where $n$ is the dimension of $V$. This function is called the ${\it representation} \hspace{0.3em}{\it zeta}\hspace{0.3em} {\it function}$ of $g \in G$.
\end{dfn}
From Definition $\ref{def2.1}$, we can say that the braid zeta function is a kind of representation zeta function.
Moreover, using the permutation representation $p_n$, we have
\begin{align*}
\zeta_{\sigma}(s)=\zeta(s, \sigma ; p_n)
\end{align*}
for $\sigma \in \mathrm{S}_n$. Hence, we can regard $\zeta_{\sigma}(s)$ as a kind of the representation zeta function.

There is a canonical projection $\pi_n : \mathrm{B}_n \longrightarrow \mathrm{S}_n$ defined by
\begin{align*}
\pi_n (\sigma_i):=(i, i+1).
\end{align*}
Since the following diagram is commutative :
\begin{equation}
\begin{CD}
\mathrm{B}_n @>\beta_{n, q} >>\mathrm{GL}_n(\Lambda) \\
@V\pi_n VV @VVq \to 1 V \\
\mathrm{S}_n @>p_n>>\mathrm{GL}_n(\Z)
\end{CD}
\end{equation}
Therefore, we immediately have
\begin{equation}
\lim_{q \to 1}\zeta(s,\sigma ; \beta_{n,q}) = \zeta_{\pi_n(\sigma)}(s).
\end{equation}
Furthermore, we can say that the following diagram is also commutative :
\begin{align*}
\begin{CD}
\mathrm{B}_n @>\mathrm{sgn}_q >>(-q)^{\Z} \\
@V\pi_n VV @VVq \to 1 V \\
\mathrm{S}_n @>\mathrm{sgn}>>\Set{\pm1} 
\end{CD}
\end{align*}

\begin{proof}[Proof of Theorem $1.2$]
$(1)$ By the definition of braid zeta function, 
\begin{align*}
\zeta(s,\sigma ; \beta_{n,q})=\det(I_n-\beta_{n,q}(\sigma)s)^{-1}
&=\det(\beta_{n,q}(\sigma)s)^{-1}(-1)^n\det(I_n-\beta^{-1}_{n,q}(\sigma)s^{-1})^{-1}\\
&=\mathrm{sgn}_q (\sigma^{-1})(-s)^{-n}\zeta(s^{-1},\sigma^{-1} ; \beta_{n,q}).
\end{align*}
$(2)$ It is well-known that the Burau representation $\beta_{n,q}$ can be decomposed into the trivial representation $\mathbf{1}$ and an $(n-1)$-dimensional irreducible representation $\beta_{n,q}^r$ :
\begin{align*}
\beta_{n,q}=\mathbf{1} \oplus \beta_{n,q}^r,
\end{align*}
where $\beta_{n,q}^r$ is defined by
\begin{align*}
\beta_{n,q}^r(\sigma_i):=
\begin{cases} \left(
\begin{array}{ccc}
-q & 1 \\
0 & 1 \\
\end{array}\right) \oplus I_{n-3} &(i=1),\\
I_{i-2}\oplus \left(
\begin{array}{cccc}
1 & 0 & 0 \\
q & -q & 1 \\
0 & 0 & 1 \\
\end{array}\right) \oplus I_{n-i-2} &(i=2, 3, \ldots, n-2),\\
I_{n-3}\oplus \left(
\begin{array}{ccc}
1 & 0 \\
q & -q \\
\end{array}\right) &(i=n-1).\\
\end{cases}
\end{align*}
This is called the ${\it reduced}\hspace{0.3em} {\it Burau} \hspace{0.3em}{\it representation}$.This fact is obtained by computing $C_{n,q}^{-1} \cdot \beta_{n,q} (\sigma_i) \cdot C_{n,q}$ (see $\cite[\mathrm{Theorem} \ 3.9]{siryo6}$), where 
\begin{align*}
C_{n,q} :=\left(
\begin{array}{ccccccc}
1 & 1 & 1 & \cdots & 1 \\
   & q & q & \cdots & 1 \\
   &    & q^2 & \cdots & q^2 \\
   &    &       & \ddots & \vdots \\
   &    &       &            & q^{n-1} \\
\end{array}\right) \in \mathrm{GL}_{n-1}(\Lambda).
\end{align*}
When $\hat{\sigma}$ is a knot, Burau proved that the Alexander polynomial can be described by using the reduced Burau representation (see \cite[Theorem 3.11]{siryo3}):
\begin{align*}
\mathrm{det}(I_{n-1}-\beta_{n,q}^r(\sigma))=(1+q+q^2+\cdots +q^{n-1})\Delta_{\hat{\sigma}}(q),
\end{align*}
where, $\Delta_{\hat{\sigma}}(q)$ is the Alexander polynomial of the knot $\hat{\sigma}$.
Since $\pi_n (\sigma)$ is the simple cycle, $\zeta (s,\pi_n (\sigma) ;p_n)$ has a simple pole at $s=1$. 
On the other hand, by the decomposition of the Burau representation, $\zeta(s, \sigma ; \beta_{n,q} )$ must have a pole at $s=1$. By the formula $(2.2)$, the order of this pole is smaller or equal to the order of the pole of $\zeta (s,\pi_n (\sigma) ;p_n)$ at $s=1$. Thus $\zeta(s, \sigma ; \beta_{n,q} )$ has a simple pole at $s=1$. Moreover the residue of $\zeta(s, \sigma ; \beta_{n,q} )$ can be calculated as follows :
\begin{align*}
\underset{s=1}{\Res} \hspace{0.2em}\zeta(s,\sigma; \beta_{n,q})
&=\lim_{s \to 1}\mathrm{det}(I_n-\mathbf{1} \oplus \beta_{n,q}^r(\sigma)s)^{-1}\\
&=\lim_{s \to 1}(s-1)\mathrm{det}(I_{n-1}-\beta_{n,q}^r(\sigma)s)^{-1}\\
&=-\frac{1}{1+q+\cdots+q^{n-1}}\mathrm{det}(I_{n-1}-\beta_{n,q}^r(\sigma)s)^{-1}\\ 
&=-\frac{1-q}{1-q^n}\Delta_{\hat{\sigma}}(q)^{-1}\\
&=-\frac{1}{[n]_q}\Delta_{\hat{\sigma}}(q)^{-1}.
\end{align*}
$(3)$ If the absolute values of the eigenvalues of $\beta_{n,q}^r (\sigma)$ are all equal to $1$, then all poles of $\zeta(e^{-s},\sigma ;\beta_{n,q})$ satisfy
\begin{align*}
e^{-\mathrm{Re}(s)}=|e^{-s}|=|\alpha_q|^{-1}=1,
\end{align*}
where $\alpha_q$ is one of the eigenvalues of $\beta_{n,q}^r (\sigma)$. Then, the real part of $s$ is equal to 0:
\begin{align*} 
\mathrm{Re}(s)=0.
\end{align*}
Hence, it is sufficient to show that the absolute values of the eigenvalues of $\beta_{n,q}^r (\sigma)$ are all equal to $1$. In \cite{siryo9}, Squier proved that the reduced Burau representation is unitary in the following sense. 
We put
\begin{align*}
\Omega_n^r=\left(
\begin{array}{ccccc}
q^{\frac{1}{2}}+q^{-\frac{1}{2}} & -q^{-\frac{1}{2}} & &O \\
-q^{\frac{1}{2}} & \ddots & \ddots & & \\
& \ddots & \ddots & -q^{-\frac{1}{2}} \\
O& & -q^{\frac{1}{2}} & q^{\frac{1}{2}}+q^{-\frac{1}{2}}
\end{array}\right)\in \mathrm{GL}_{n-1}(\Z[q^{\pm\frac{1}{2}}]).
\end{align*}
Then, the following equation holds for any braid $\sigma \in \mathrm{B}_n$. 
\begin{align*}
{\beta_{n,q}^r(\sigma)} \cdot \Omega_n^r \cdot \beta_{n,q}^r(\sigma)^* = \Omega_n^r.
\end{align*}
Here for a matrix $A \in \mathrm{GL}_{n-1}(\Z[q^{\pm \frac{1}{2}}])$, $A^*$ is the conjugate-transpose of $A$ and the conjugate of $a(q) \in \Z[q^{\pm\frac{1}{2}}]$ is defined to be $a(q^{-1})$. Since $q$ belongs to the unit circle on $\C$, we can regard that  $\mathrm{B}_n$ acts on $\C^{n-1}$ by using the reduced Burau representation. 
\begin{align*}
\beta_{n,q}^r : \mathrm{B}_n \longrightarrow \mathrm{GL}_{n-1}(\C).
\end{align*}
Then $A^*$ coincides with the conjugate-transpose of the matrix $A$ with complex entries. Now, we define the following quadratic form for two row vectors $\mathbf{u},\mathbf{v} \in \C^{n-1}$
\begin{align*}
\langle \mathbf{u}, \mathbf{v} \rangle_{\mathrm{B}_n} := \mathbf{u} \cdot \Omega_n^r \cdot \mathbf{v}^*.
\end{align*}
Here $\mathbf{v}^*$ is the conjugate-transpose of the row vector $\mathbf{v}$. Then, we have 
\begin{align*}
\langle \mathbf{u}\beta_{n,q}^r(\sigma), \mathbf{v}\beta_{n,q}^r(\sigma) \rangle_{\mathrm{B}_n} 
&=\mathbf{u} \beta_{n,q}^r(\sigma) \cdot \Omega_n^r \cdot \beta_{n,q}^r(\sigma)^* \mathbf{v}^* \\
&=\mathbf{u} \cdot \Omega_n^r \cdot \mathbf{v}^* \\
&=\langle \mathbf{u}, \mathbf{v} \rangle_{\mathrm{B}_n}.
\end{align*}
Since $\Omega_n^r$ is the Hermitian matrix, the quadratic form $\langle  \cdot ,\cdot  \rangle_{\mathrm{B}_n}$ is positive definite if and only if the eigenvalues of $\Omega_n^r$ are all positive. In this case, the eigenvalues of $\Omega_n^r$ can be computed explicitly by using the formula for the tridiagonal matrix as below (see \cite{siryo10}), that is, a set of eigenvalues of $\Omega_n^r$ coincides with,
\begin{align*}
\bigl\{ q^{\frac{1}{2}}+q^{-\frac{1}{2}}-2\cos{\frac{\pi j}{n}} \mid (j=1, 2, \ldots, n-1)\bigr\}.
\end{align*}
Hence, consequently we can say that $\langle  \cdot  ,\cdot  \rangle_{\mathrm{B}_n}$ is positive definite if and only if
\begin{equation}
|\theta|<\frac{2\pi}{n}.
\end{equation}
Let $\mathbf{u}$ be an eigenvector of $\beta_{n,q}^r (\sigma)$ with the eigenvalue $\alpha_q$ for $\sigma \in \mathrm{B}_n$. Then we have
\begin{align*}
\langle \mathbf{u}, \mathbf{u} \rangle_{\mathrm{B}_n} = \langle \mathbf{u}\beta_{n,q}^r(\sigma), \mathbf{u}\beta_{n,q}^r(\sigma) \rangle_{\mathrm{B}_n} &= \langle \alpha_q \mathbf{u}, \alpha_q \mathbf{u} \rangle_{\mathrm{B}_n}=|\alpha_q|^2\langle \mathbf{u}, \mathbf{u} \rangle_{\mathrm{B}_n}.
\end{align*}
Under the condition $(2.4)$, all eigenvalues of $\beta_{n,q}^r (\sigma)$ satisfy $|\alpha_q |=1$. Therefore we complete the proof of $(3)$.
\end{proof}
We give some examples of braid zeta function.
\begin{ex}\normalfont
Let $\sigma=\sigma_1^3 \in \mathrm{B}_2$.
The matrix $\beta_{2,q}(\sigma_1^3)$ is presented by
\begin{align*}
\beta_{2,q}(\sigma_1^3)=\left(
\begin{array}{ccc}
1-q+q^2-q^3 & 1-q+q^2 \\
q(1-q+q^2) & q(1-q) \\
\end{array}\right).
\end{align*}
Thus, we have
\begin{align*}
\zeta(s,\sigma_1^3 ; \beta_{2,q})&=\mathrm{det}\left(
\begin{array}{ccc}
1-(1-q+q^2-q^3)s & -(1-q+q^2)s \\
-q(1-q+q^2)s & 1-q(1-q)s \\
\end{array}\right)^{-1}\\
&=\frac{1}{(1-s)(1+q^3s)}.
\end{align*}
Then, we can compute the residue of $\zeta(s,\sigma_1^3 ; \beta_{2,q})$ as
\begin{align*}
\underset{s=1}{\Res}\hspace{0.2em}\zeta(s,\sigma_1^3 ; \beta_{2,q})&=\lim_{s \to 1}(s-1)\frac{1}{(1-s)(1+q^3s)}\\
&=-\frac{1}{1+q^3}\\
&=-\frac{1}{[2]_q}\Delta_{\widehat{\sigma_1^3}}(q)^{-1}.
\end{align*}
Here remark that the closure of $\sigma_1^3$ is illustrated in Figure 3.
\begin{figure}[!h]
\begin{center}
\includegraphics[,width=3.0cm]{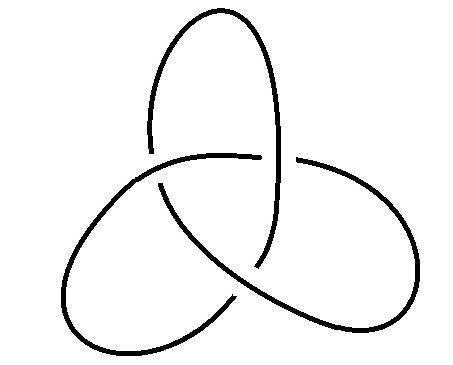} 
\caption{$\widehat{\sigma_1^3}$}
\end{center}
\end{figure}
This is called the ${\it trefoil}$\hspace{0.3em}${\it knot}$.
\end{ex}
\begin{ex}\normalfont
Let $\sigma=(\sigma_1\sigma_2^{-1})^2$.
The matrix $\beta_{3,q} ((\sigma_1\sigma_2^{-1})^2)$ is calculated as
\begin{align*}
\beta_{3,q} ((\sigma_1\sigma_2^{-1})^2)=\left(
\begin{array}{cccc}
(1-q)^2 & q^{-1} & -(1-q^{-1})^2 \\
q(1-q) & 0 & 1 \\
q & 1-q^{-1} & (1-q^{-1})^2 \\
\end{array}\right).
\end{align*}
Then we have
\begin{align*}
\zeta(s,(\sigma_1\sigma_2^{-1})^2 ; \beta_{3,q})=\frac{q^2}{(1-s)(q^2-(1-2q+q^2-2q^3+q^4)s+q^2s^2)}.
\end{align*}
Hence
\begin{align*}
\underset{s=1}{\Res}\hspace{0.2em} \zeta(s,(\sigma_1\sigma_2^{-1})^2; \beta_{3,q})
&=\lim_{s \to 1}\frac{q^2}{q^2-(1-2q+q^2-2q^3+q^4)s+q^2s^2}\\
&=-\frac{q^2}{(1+q+q^2)(-1+3q-q^2)}\\
&=-\frac{1}{[3]_q}\Delta_{\widehat{(\sigma_1\sigma_2^{-1})^2}}(q)^{-1}.
\end{align*}
\newpage
The closure $\widehat{(\sigma_1\sigma_2^{-1})^2}$ is called the ${\it figure}$-${\it eight}$\hspace{0.3em}${\it knot}$, which is illustrated in Figure 4. 
\begin{figure}[!h]
\begin{center}
\includegraphics[,width=3.5cm]{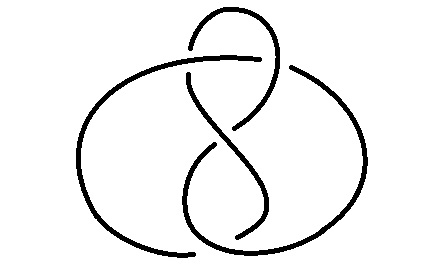} 
\caption{$\widehat{(\sigma_1\sigma_2^{-1})^2}$}
\end{center}
\end{figure}
\end{ex}
Comparing (1.1) with Theorem$1.2\hspace{0.2em}(2)$, we have the following fact.
\begin{cor}
If $\hat{\sigma}$ is a knot for $\sigma \in \mathrm{B}_n$. Then, the Alexander polynomial $\Delta_{\hat{\sigma}}$ holds
\begin{equation}
\Delta_{\hat{\sigma}}(1)=1.\\
\end{equation}
\end{cor}
\begin{proof}
By  Theorem $1.2\hspace{0.2em}(2)$, 
\begin{align*}
\lim_{q \to 1}\underset{s=1}{\Res}\hspace{0.2em} \zeta(s, \sigma ; \beta_{n,q})&=-\lim_{q \to 1}\frac{1}{[n]_q}\Delta_{\hat{\sigma}}(q)^{-1}\\
&=-\frac{1}{n}\Delta_{\hat{\sigma}}(1)^{-1}.
\end{align*}
Then, we obtain $(2.4)$ by the commutative diagram $(2.1)$.
\end{proof}
Under the condition $(2.3)$, we obtain the following generating function expression which converges when the absolute value of $s$ is smaller than $1$. 
\begin{equation}
\zeta(s,\sigma ; \beta_{n,q})=\exp\biggl\{\sum_{m=1}^{\infty}\frac{\mathrm{tr}\beta_{n,q}(\sigma^m)}{m}s^m\biggr\}.
\end{equation}
Later, we assume that $q$ satisfies the condition $(2.3)$ for $\sigma \in \mathrm{B}_n$ and that $q$ is not a root of unity for simplicity. 
By using the expression $(2.5)$, we obtain the following formula.
\begin{prp}
For any braid $\sigma \in \mathrm{B}_n$, 
\begin{align*}
\left.\frac{d}{ds}\log \zeta(s,\sigma ;\beta_{n,q})\right|_{s=0}=\mathrm{tr}\beta_{n,q}(\sigma).
\end{align*}
\end{prp}
\begin{proof}
From the expression $(2.5)$,
\begin{align*}
\left.\frac{d}{ds}\log \zeta(s,\sigma ;\beta_{n,q})\right|_{s=0}&=\left.\frac{d}{ds}\sum_{m=1}^{\infty}\frac{\mathrm{tr}\beta_{n,q}(\sigma^m)}{m}s^m\right|_{s=0}\\
&=\left.\sum_{m=1}^{\infty}\mathrm{tr}\beta_{n,q}(\sigma^m)s^{m-1}\right|_{s=0}\\
&=\mathrm{tr}\beta_{n,q}(\sigma).
\end{align*}
\end{proof}
\section{Some formulae for the torus type braid}

~~~In this section, we give the explicit formula of the zeta function for the special braid whose closure is isotopic to a torus knot. Furthermore, we define the function $Z_q$ of several braids by using the Kronecker tensor product. If the braids are all the same, then $Z_q$ is equal to the zeta function associated with the tensor product representation $\beta_{n, q}^{\otimes r}$. We consider the case of same torus type braids (Theorem $3.2$) and the case of distinct torus type braids (Theorem $3.3$). Using the generating function expression, we show that the function $Z_q$ can be written by some braid zeta functions for each case. As a corollary, the reside of $Z_q$ at $s=1$ can be expressed by some Alexander polynomials. 
First of all we give the definition of the torus type braid. 
\begin{dfn}\normalfont
For a coprime pair $(n,m) \in \N \times \Z$, we define 
\begin{align*}
\sigma_{n,m}:=(\sigma_1\sigma_2\cdots\sigma_{n-1})^m \in \mathrm{B}_n.
\end{align*}
Then the closure of $\sigma_{n,m}$ is isotopic to the torus knot $T(n,m)$. We call $\sigma_{n,m}$ the ${\it torus}$\hspace{0.3em}${\it type}$\hspace{0.3em}${\it braid}$.
\end{dfn}
Since $\mathrm{B}_1$ is trivial, we assume that the number of strands $n$ is larger than 1 in this section.
\begin{ex}\normalfont
The closure of $\sigma_{2,3} =\sigma_1^3 \in \mathrm{B}_2$ is the torus knot $T(2,3)$.
\end{ex}
\begin{ex}\normalfont
The braid $\sigma_{3,5} = (\sigma_1\sigma_2)^5 \in \mathrm{B}_3$ and its closure are illustrated in Figure $5$.
\begin{figure}[!h]
\begin{center}
\includegraphics[,width=7.0cm]{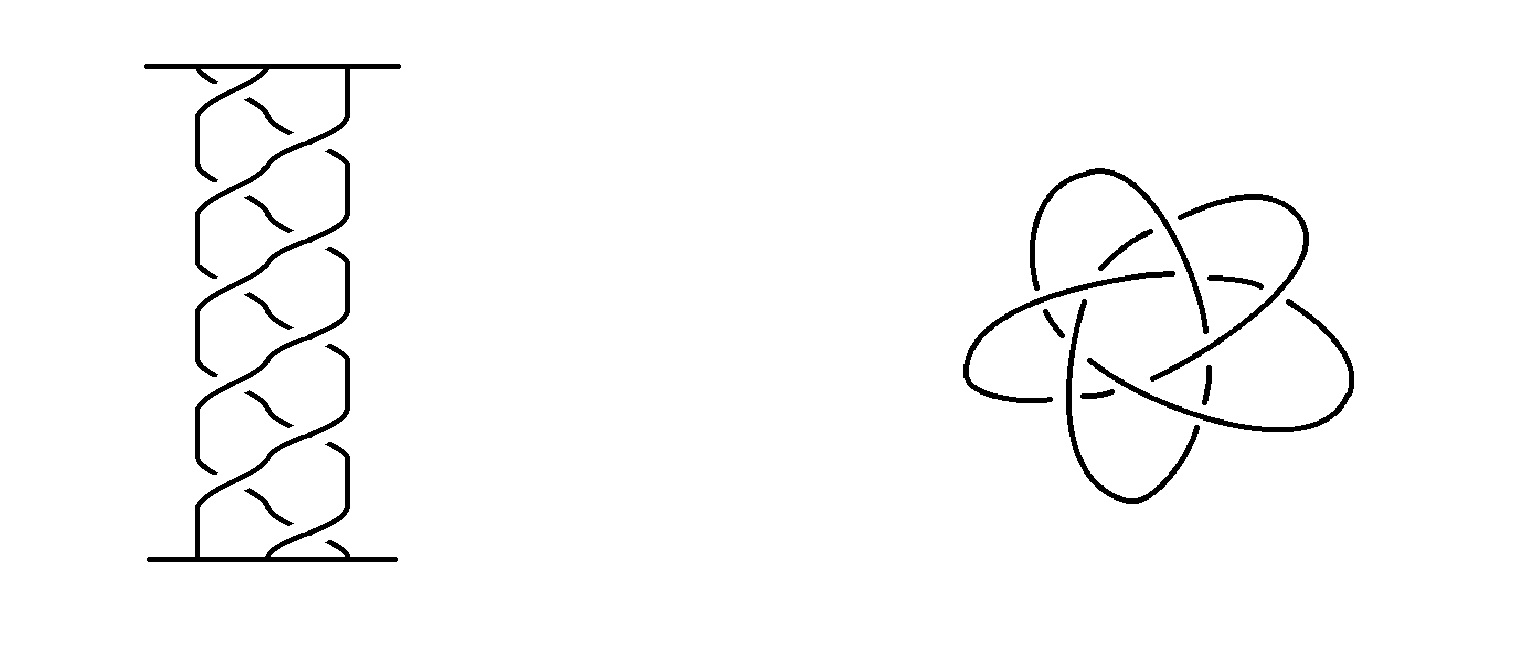} 
\caption{$\sigma_{3,5}$ and its closure}
\end{center}
\end{figure}
\end{ex}
\begin{thm}
For a coprime pair $(n,m)$, we obtain the following explicit formula:
\begin{align*}
\zeta(s,\sigma_{n,m} ;\beta_{n,q})=\frac{(1-q^ms)}{(1-s)(1-q^{nm}s^n)}.\\
\end{align*}
\end{thm}
\begin{proof}
We compute the eigenvalues of $\beta_{n,q}(\sigma_{n,m})$ over $\C[q^{\pm1}]$. By the definition of the Burau representation, 
 we have
\begin{align*}
\beta_{n,q}(\sigma_{n,1})=\left(
\begin{array}{ccccccc}
1-q & 1-q &  \cdots & 1-q & 1 \\
  q & 0 &  & &  \\
   &   q & \ddots  &  \\
   &    &  \ddots  & 0 &  \\
   &    &       &   q  & 0 \\
\end{array}\right).
\end{align*}
Thus we compute
\begin{align*}
\mathrm{det}(I_n-\beta_{n,q}(\sigma_{n,1})s)&=\mathrm{det}\left(
\begin{array}{ccccccc}
1-(1-q)s & -(1-q)s &  \cdots & -(1-q)s & -s \\
  -qs & 1 &  & &  \\
   &   -qs & \ddots  &  \\
   &    &  \ddots  & 1 &  \\
   &    &       &   -qs  & 1 \\
\end{array}\right)\\
&=(1-(1-q)s)-(1-q)s\sum_{j=1}^{n-2}(qs)^j-s(qs)^{n-1}\\
&=-s\frac{1-(qs)^n}{1-qs}+\frac{1-(qs)^{n}}{1-qs}\\
&=\frac{(1-s)(1-(qs)^n)}{1-qs}.
\end{align*}
Hence, put $\xi_n=e^{\frac{2\pi i}{n}}$, and the eigenvalues of matrix $\beta_{n,q}(\sigma_{n,1})$ are presented by $1,q^{-1} \xi_n,\ldots,q^{-1} \xi_n^{n-1}$. When a pair $(n,m)$ is coprime, the eigenvalues of $\beta_{n,q}(\sigma_{n,m})$ coincide with $1,q^{-m}\xi_n,\ldots,q^{-m}\xi_n^{n-1}$. 
Then we complete the proof of Theorem $3.1$.
\end{proof}
By Theorem $3.1$, $\zeta(s,\sigma_{n,m} ;\beta_{n,q})$ satisfies an analogue of the Riemann hypothesis without the condition $(2.3)$. Furthermore, we obtain the formula of the Alexander polynomial of the torus knot, which is the classical result.
\begin{cor}
The Alexander polynomial of the torus knot $T(n,m)$ is given by
\begin{align*}
\Delta_{T(n, m)}(q)=\frac{(1-q)(1-q^{nm})}{(1-q^n)(1-q^m)}.
\end{align*}
\end{cor}
\begin{proof}
By Theorem $1.2$ and $3.1$, we compute
\begin{align*}
\Delta_{T(n, m)}(q)^{-1}&=-[n]_q\underset{s=1}{\Res} \hspace{0.2em}\zeta(s,\sigma_{n,m}; \beta_{n,q})\\[0.5em]
&=\frac{(1-q^n)(1-q^m)}{(1-q)(1-q^{nm})}.
\end{align*}
\end{proof}
\begin{cor}
If the pair $(n,m)$ is coprime, then for $m' \in \Z$ such that $(n,m+m')$ is coprime, we have the following equation:
\begin{align*}
(1-s)\zeta(s,\sigma_{n,m+m'}; \beta_{n,q})=(1-q^{m'}s)\zeta(q^{m'}s,\sigma_{n,m}; \beta_{n,q}).\\
\end{align*}
\end{cor}
\begin{proof}
By Theorem $3.1$, we compute
\begin{align*}
\zeta(s,\sigma_{n,m+m'}; \beta_{n,q})&=\frac{(1-q^{m+m'}s)}{(1-s)(1-q^{n(m+m')}s^n)}\\
&=\frac{(1-q^m(q^{m'}s))}{(1-s)(1-q^{nm}(q^{m'}s)^n)}\\
&=\frac{(1-q^{m'}s)}{(1-s)}\cdot \frac{(1-q^m(q^{m'}s))}{(1-q^{m'}s)(1-q^{nm}(q^{m'}s)^n)}\\
&=\frac{(1-q^{m'}s)}{(1-s)}\zeta(q^{m'}s,\sigma_{n,m}; \beta_{n,q}).
\end{align*}
\end{proof}
By Corollary $3.2$, we obtain the following immediately.
\begin{cor}
If the pair $(n,m)$ is coprime, then for $m' \in \Z$ such that $(n,m+m')$ is coprime, then we have
\begin{align*}
\Delta_{T(n,m+m')}(q)=\frac{1}{[n]_q(1-q^{m'})}\zeta(q^{m'}, \sigma_{n,m}; \beta_{n,q})^{-1}.\\
\end{align*}
\end{cor}
\begin{cor}
If a pair $(n,m)$ is coprime, the following holds.
\begin{align*}
\mathrm{tr}\beta_{n,q}(\sigma_{n,m})=1-q^m.
\end{align*}
\end{cor}
\begin{proof} 
Using Proposition $2.1$ and Theorem $3.1$, we compute,
\begin{align*}
\mathrm{tr}\beta_{n,q}(\sigma_{n,m})&=\left.\frac{d}{ds}\log \frac{(1-q^ms)}{(1-s)(1-q^{nm}s^n)} \right|_{s=0}\\
&=\left.\frac{d}{ds}\bigl\{ \log(1-q^ms)-\log(1-s)-\log(1-q^{nm}s^n) \bigr\}\right|_{s=0}\\
&=\left.\bigl\{ \frac{-q^{m}}{1-q^ms}-\frac{-1}{1-s}-\frac{-n q^{nm}s^{n-1}}{1-q^{nm}s^n}\bigr\}\right|_{s=0}\\
&=1-q^m.
\end{align*}
\end{proof}
Now, we define a function associated with some braids. 
\begin{dfn}\normalfont
For $\tau_{1} \in \mathrm{B}_{n_1}, \ldots ,\tau_r \in \mathrm{B}_{n_r}$, we define 
\begin{align*}
Z_q(s ; \tau_1, \ldots,  \tau_r):= \mathrm{det}(I_{n_1\cdots n_r}-\beta_{n_1,q}(\tau_1)\otimes \cdots \otimes \beta_{n_r,q}(\tau_r)s)^{-1}.
\end{align*}
Here $\otimes$ is the Kronecker tensor product of matrices. Especially, for a positive integer $r$ and $\sigma \in \mathrm{B}_n$ we can denote
\begin{align*}
Z_q(s ; 	\underbrace{\sigma, \ldots, \sigma}_{r} )=\zeta(s, \sigma ; \beta_{n,q}^{\otimes r}).
\end{align*}
\end{dfn}
Let $X, Y$ be the finite sets with $\#X=n, \#Y=m$. In \cite{siryo7}, 
for $\sigma \in \mathrm{S}_n = \mathrm{Aut}(X), \tau \in \mathrm{S}_m = \mathrm{Aut}(Y)$, we define $\zeta_{\sigma \otimes \tau }(s)$ as the zeta function of a dynamical system over $X\times Y$. By the commutative diagram $(2.1)$,  we have
\begin{align*}
\lim_{q \to 1}Z_q(s,\sigma, \tau) = \zeta_{\pi_n(\sigma) \otimes \pi_m(\tau)}(s),
\end{align*}
for $\sigma \in \mathrm{B}_n, \tau \in \mathrm{B}_m$. Then, we can regard $Z_q(s,\sigma, \tau)$ as the $q$-analogue of $\zeta_{\pi_n(\sigma) \otimes \pi_m(\tau)}(s)$. Now, we give some formulae for the torus type braids. 
\begin{thm}
We choose $n \in \Z_{\geq 2}, m \in \Z, r \in \N$ such that the pair $(n, r!\cdot m)$ is coprime. 
We put
\begin{align*}
K_{n,m,r}(s,q):=\prod_{l=1}^{r}\biggl( \frac{1-s}{1-q^{lm}s}\biggr)^{a_{r,l}+b_{r,l,n}}.
\end{align*}
Here $a_{r,l}:={}_r C _l (-1)^l, b_{r,l,n}:={}_r C _l \frac{(n-1)^l-(-1)^l}{n}$. Then we have
\begin{equation}
\zeta(s, \sigma_{n,m} ; \beta_{n,q}^{\otimes r})=K_{n,m,r}(s,q) \cdot \prod_{l=1}^{r}\zeta(s,\sigma_{n,lm} ; \beta_{n,q})^{b_{r,l,n}}.
\end{equation}
\end{thm}
\begin{proof}
Since we know the eigenvalues of $\beta_{n,q}(\sigma_{n,m})$, $\mathrm{tr}\beta_{n,q}(\sigma_{n,m}^j)$ can be calculated directly by
\begin{align*}
\mathrm{tr}\beta_{n,q}(\sigma_{n,m}^j)&=1+q^{jm}\xi_n^j +\cdots +q^{jm}\xi_n^{j(n-1)}\\[0.5em]
&=1+q^{jm}(1+\xi_n^j+\cdots+\xi_n^{j(n-1)})-q^{jm}\\[0.5em]
&=\begin{cases}
1-q^{mj}
& j \not \equiv0 \pmod n, \\
1-q^{mkn}+nq^{mkn}
& j=kn \hspace{0.5em}(k \in \N). 
\end{cases}
\end{align*}
Note that this is another proof of Corollary $3.4$ when $j=1$. Then, we compute
\begin{align*}
\mathrm{tr}(\beta_{n, q}^{\otimes r}(\sigma_{n,m})^j)&=\mathrm{tr}((\beta_{n,q}(\sigma_{n,m})\otimes \cdots \otimes \beta_{n,q}(\sigma_{n,m}))^j)\\[0.5em]
&= \mathrm{tr}(\beta_{n,q}(\sigma_{n,m}^j))^r\\[0.5em]
&=\begin{cases}
(1-q^{mj})^r
& \hspace{0.6em}j\not \equiv0 \hspace{-0.6em}\pmod n, \\
(1+(n-1)q^{mkn})^r
&\hspace{0.6em} j=kn \hspace{0.5em}(k \in \N), 
\end{cases}\\[0.3em]
&=\begin{cases}
\displaystyle\sum_{l=0}^r {}_r C _l (-1)^lq^{lmj}
& j\not \equiv0\hspace{-0.6em}\pmod n, \\[1.0em]
\displaystyle\sum_{l=0}^r {}_r C _l (n-1)^l q^{lmkn}
& j=kn \hspace{0.5em}(k \in \N). 
\end{cases}
\end{align*}
Hence,
\begin{align*}
\zeta(s, \sigma_{n,m} ; \beta_{n,q}^{\otimes r})&=\exp \biggl\{ \sum_{j=1}^{\infty}\sum_{l=0}^r {}_r C _l (-1)^l \frac{q^{lmj}}{j}s^j +\sum_{k=1}^{\infty} \sum_{l=0}^r {}_r C _l \frac{(n-1)^l-(-1)^l}{n} \cdot \frac{q^{lmkn}}{k}s^{kn}\biggr\}\\
&=\exp \biggl\{ \sum_{l=0}^r \{ {}_r C _l (-1)^l \log(1-q^{lm}s)^{-1}+{}_r C _l \frac{(n-1)^l-(-1)^l}{n} \log(1-q^{lmn}s^n)^{-1}\}\biggr\}\\
&=\prod_{l=0}^r \exp \biggl\{a_{r,l}\log(1-q^{lm}s)^{-1}+b_{r,l,n}\log(1-q^{lmn}s^n)^{-1}\biggr\}\\
&=\prod_{l=0}^{r}(1-q^{lm}s)^{-a_{r,l}}(1-q^{lmn}s^{n})^{-b_{r,l,n}}.\\
\end{align*}
On the other hand, we calculate the right-hand side of $(3.1)$ as
\begin{align*}
K_{n,m,r}(s,q) \prod_{l=1}^{r}\zeta(s,\sigma_{n,lm} ; \beta_{n,q})^{b_{r,l,n}}
&=\prod_{l=1}^{r}\biggl( \frac{1-s}{1-q^{lm}s}\biggr)^{a_{r,l}+b_{r,l,n}}\biggl( \frac{1-q^{lm}s}{(1-s)(1-q^{lmn}s^n)}\biggr)^{b_{r,l,n}}\\
&=\prod_{l=1}^{r}\biggl(\frac{1-s}{1-q^{lm}s}\biggr)^{a_{r,l}}(1-q^{lmn}s^n)^{-b_{r,l,n}}\\
&=(1-s)^{a_{r,1}+a_{r,2}+\cdots +a_{r,r}}\prod_{l=1}^{r}(1-q^{lm}s)^{-a_{r,l}}(1-q^{lmn}s^n)^{-b_{r,l,n}}\\
&=(1-s)^{-1}\prod_{l=1}^{r}(1-q^{lm}s)^{-a_{r,l}}(1-q^{lmn}s^n)^{-b_{r,l,n}}\\
&=\prod_{l=0}^{r}(1-q^{lm}s)^{-a_{r,l}}(1-q^{lmn}s^n)^{-b_{r,l,n}}.
\end{align*}
This completes the proof of Theorem $3.2$. 
\end{proof}
\begin{cor}
We choose $n \in \Z_{\geq 2}, m \in \Z, r \in \N$ such that the pair $(n, r!\cdot m)$ is coprime. Then we have
\begin{equation}
\underset{s=1}{\Res}\hspace{0.2em} \zeta(s, \sigma_{n,m} ; \beta_{n,q}^{\otimes r}) = -\frac{1}{[n]_q^{n^{r-1}}}\prod_{l=1}^r\frac{1}{(1-q^{lm})^{a_{r,l}+b_{r,l,n}}\Delta_{T(n,lm)}(q)^{b_{r,l,n}}}.\\
\end{equation}
\end{cor}
\begin{proof}
The number of $[n]_q$ is calculated by
\begin{align*}
\sum_{l=1}^rb_{r,l,n}=\sum_{l=0}^rb_{r,l,n}&=\frac{1}{n}\sum_{l=0}^r{}_r C _l (n-1)^l-\frac{1}{n}\sum_{l=0}^r{}_r C _l (-1)^l\\
&=\frac{1}{n}(1+(n-1))^r\\[1.0em]
&=n^{r-1}.
\end{align*}
Hence the formula $(3.2)$ follows from Theorem $1.2$ and $3.2$.
\end{proof} 
In general, if we put $n=(r\cdot|m|)!+1$, the pair $(n, lm)$ is coprme for each $1\leq l \leq r$. This gives an example of Corollary $3.5$.
\begin{ex}\normalfont
We show the case of $r=2$. If a positive odd integer $n$ is coprime to $m \in \Z$, then we compute
\begin{align*}
\mathrm{tr}((\beta_{n,q}(\sigma_{n,m})\otimes \beta_{n,q}(\sigma_{n,m}))^j)&=\mathrm{tr}(\beta_{n,q}(\sigma_{n,m}^j)\otimes \beta_{n,q}(\sigma_{n,m}^j))\\[0.5em]
&=\mathrm{tr}(\beta_{n,q}(\sigma_{n,m}^j))^2\\[0.5em]
&=\begin{cases}
(1-q^{mj})^2
& j \not \equiv0 \pmod n, \\[0.3em]
(1-q^{mkn}+nq^{mkn})^2
& j=kn \hspace{0.5em}(k \in \N).
\end{cases}
\end{align*}
Hence, 
\begin{align*}\hspace{-2.0em}
\zeta(s, \sigma_{n,m} ; \beta_{n,q}^{\otimes 2})&=
\exp \biggl\{ \sum_{j=1}^{\infty}\frac{\mathrm{tr}((\beta_{n,q}(\sigma_{n,m})\otimes \beta_{n,q}(\sigma_{n,m}))^j)}{j}s^j \biggr\} \\
&=\exp\biggl\{ \sum_{j \not \equiv 0\hspace{-0.6em}\pmod n}\frac{1-2q^{mj}+q^{2mj}}{j}s^j +\sum_{k=1}^{\infty}\frac{1-2q^{mkn}+q^{2mkn}+2nq^{mkn}-2nq^{2mkn}+n^2q^{2mkn}}{kn}s^{kn}\biggr\} \\
&=\exp\biggl\{ \sum_{j=1}^{\infty}\frac{1-2q^{mj}+q^{2mj}}{j}s^j+\sum_{k=1}^{\infty}\frac{2q^{mkn}-2q^{2mkn}+nq^{2mkn}}{k}s^{kn}\biggr\} \\
&=\frac{(1-q^ms)^2}{(1-s)(1-q^{2m}s)(1-q^{mn}s^n)^2(1-q^{2mn}s^n)^{n-2}}\\[0.7em]
&=(\frac{1-s}{1-q^{2m}s})^{n-1}\biggl\{\frac{(1-q^ms)}{(1-s)(1-q^{mn}s^n)}\biggr\}^{2}\biggl\{\frac{1-q^{2m}s}{(1-s)(1-q^{2mn}s^n)}\biggr\}^{n-2}\\[1.0em]
&=(\frac{1-s}{1-q^{2m}s})^{n-1}\zeta(s,\sigma_{n,m};\beta_{n,q})^2\zeta(s,\sigma_{n,2m};\beta_{n,q})^{n-2}.\\
\end{align*}
By Theorem $1.2$, we obtain the following.
\begin{align*}
\underset{s=1}{\Res}\hspace{0.2em} \zeta(s,\sigma_{n,m}; \beta_{n,q}^{\otimes 2})=-\frac{1}{(1-q^{2m})^{n-1}[n]_q^n\Delta_{T(n,m)}^2(q)\Delta_{T(n,2m)}^{n-2}(q)}.\\
\end{align*}
\end{ex}
Next, we define some symbols to prove the formula of the case of some distinct torus type braids. 
\begin{dfn}\normalfont
$(1)$ We put $\Omega :=\Set{1, 2, \ldots , r}$. For $\mathbf{n}=\Set{n_i \in \Z \mid i \in I \subseteq \Omega }$, we define
\begin{align*}
[\mathbf{n}]&:=\prod_{i \in I} n_i,\hspace{1.0em} [\Set{\emptyset}]:=1, \\[0.5em]
|\mathbf{n}|&:=\sum_{ i \in I} n_i,\hspace{1.0em}|\Set{\emptyset}|:=0.
\end{align*}
$(2)$ We assume that $n_1, \ldots, n_r \in \Z_{>1}$ are all coprime. Then for $I \subseteq \Omega$, we define
\begin{align*}
A_i&:=n_i \N=\Set{k \in \N \mid k \equiv 0 \pmod{n_i}} \hspace{0.5em}(1\leq i\leq r),\\[1.0em]
E_I&:=(\bigcap_{i \in I}A_i)\cap(\bigcap_{j \notin I}A_j^c),\\[0.8em]
F_I&:=\bigcap_{i \in I}A_i =\bigcup_{I\subseteq I'}E_{I'}.
\end{align*}
$(3)$ We define the subset associated with an index set $I \subseteq \Omega$ as follows:
\begin{align*}
\mathbf{n}(I):=\Set{n_{i}\in \Z_{>1} \mid i \in I},\hspace{1.0em} \mathbf{m}(I):=\Set{m_{i}\in \Z \mid  i \in I}.
\end{align*}
\end{dfn}
\begin{rem}\normalfont
Remark that the following holds:
\begin{align*}
\N=\bigsqcup_{I \subseteq \Omega}E_I.
\end{align*}
\end{rem}
We give the following lemma.
\begin{lem}
$(1)$ We fix the subset $I \subseteq \Omega$, then we have
\begin{align*}
\sum_{I\subseteq J \subseteq \Omega}(-1)^{\# J}=(-1)^{\# I}\delta_{I,\Omega}.
\end{align*}
Here for subsets $I, J \subset \Omega$, $\delta_{I,J}$ is defined as follows:
\begin{align*}
\delta_{I,J}:=\begin{cases}
1&I=J,\\[0.5em]
0&I\neq J.
\end{cases}
\end{align*}
$(2)$ We have the following formula:
\begin{align*}
\sum_{\emptyset \neq I \subseteq J \subseteq \Omega}(-1)^{\#J-\#I}=1.
\end{align*}
\end{lem}
\begin{proof}\normalfont
$(1)$ We compute
\begin{align*}
\sum_{I \subseteq J \subseteq \Omega}(-1)^{\#J}&=\sum_{t=0}^{\# \Omega -\# I}{}_{\# \Omega -\# I} C_{t}(-1)^{t+\#I}\\
&=(-1)^{\# I}\sum_{t=0}^{\# \Omega -\# I}{}_{\# \Omega -\# I} C_{t}(-1)^{t}\\[0.5em]
&=
(-1)^{\# I}\delta_{I,\Omega}.
\end{align*}
$(2)$ By $(1)$, we have
\begin{align*}
\sum_{\emptyset \neq I \subseteq J \subseteq \Omega}(-1)^{\#J-\#I}&=\sum_{\emptyset \neq I \subseteq \Omega}(-1)^{- \# I}\sum_{ I \subseteq J \subseteq \Omega}(-1)^{\# J}\\[0.5em]
&=\sum_{\emptyset \neq I \subset \Omega}(-1)^{-\# I}\cdot 0+(-1)^{-\# \Omega}\cdot (-1)^{\# \Omega}\\[0.5em]
&=1.\\
\end{align*}
Hence we complete the proof of Lemma $3.1$.
\end{proof}
\begin{dfn}
For subset $I \subseteq \Omega$, we define the followings:
\begin{align*}
T^{(1)}_I(z)&:=\prod_{i \in I}(1-z^{m_i}+n_iz^{m_i}),\\[0.5em]
T^{(2)}_I(z)&:=\prod_{i \in I}(1-z^{m_i}).\\
\end{align*}
\end{dfn}
\begin{thm}
If the pairs $([\mathbf{n}(I)], |\mathbf{m}(I)|)\hspace{0.2em}$ are all coprime for any $I \subseteq \Omega$,
and $n_1, \ldots, n_r$ are also coprime, then we obtain the following formula:
\begin{align*}
Z_q(s; \sigma_{n_1,m_1}, \ldots, \sigma_{n_r, m_r})=\prod_{\emptyset \neq I \subseteq J \subseteq \Omega}\zeta(s, \sigma_{[\mathbf{n}(I)], |\mathbf{m}(J)|}; \beta_{[\mathbf{n}(I)],q})^{(-1)^{\#J-\# I}}.
\end{align*}
\end{thm}
\begin{proof}
If $j \in E_I$, 
\begin{align*}
\mathrm{tr}(\beta_{n_i, q}(\sigma_{n_i,m_i}^j))=
\begin{cases}
(1-q^{m_i j} +n_iq^{m_i j})&i \in I,\\[0.5em]
(1-q^{m_i j})& i \not \in I.
\end{cases}
\end{align*}
Then we can write
\begin{align*}
\mathrm{tr}(\beta_{n_1,q}(\sigma_{n_1,m_1}^j)\otimes \cdots \otimes \beta_{n_r, q}(\sigma_{n_r, m_r}^j))=\mathrm{tr}(\beta_{n_1,q}(\sigma_{n_1,m_1}^j)) \cdots\mathrm{tr}( \beta_{n_r, q}(\sigma_{n_r, m_r}^j))=T^{(1)}_I(q^j)\cdot T^{(2)}_{\Omega \setminus I}(q^j).
\end{align*}
For $I \subseteq \Omega$, $T^{(1)}_I(q^j)$ and $T^{(2)}_I(q^j)$ are calculated as
\begin{eqnarray}
T^{(1)}_I(q^j)&=\displaystyle{\sum_{I' \subseteq I }}[\mathbf{n}(I')] \cdot q^{ |\mathbf{m}(I')|  j}T^{(2)}_{I'}(q^j),\\[1.0em]
T^{(2)}_I(q^j)&\hspace{-4.2em}=\displaystyle{\sum_{I' \subseteq I }}(-1)^{\#I'}q^{|\mathbf{m}(I')| j}.
\end{eqnarray}
Then we compute
\begin{align*}
Z_q(s; \sigma_{n_1,m_1}, \ldots, \sigma_{n_r,m_r})
&=\exp\biggl\{\sum_{I \subseteq \Omega}\sum_{j \in E_I}\frac{T^{(1)}_I(q^j)\cdot T^{(2)}_{\Omega \setminus I}(q^j)}{j}s^j \biggr\}\\[0.7em]
&=\exp\biggl\{\sum_{I \subseteq \Omega}\sum_{j \in E_I}\sum_{I' \subseteq I } \frac{[\mathbf{n}(I')]\cdot q^{|\mathbf{m}(I')|  j}T^{(2)}_{I \setminus I'}(q^j)\cdot T^{(2)}_{\Omega \setminus I}(q^j)}{j}s^j \biggr\}\\[0.7em]
\end{align*}
\begin{align*}
&=\exp\biggl\{\sum_{I \subseteq \Omega}\sum_{j \in E_I}\sum_{I' \subseteq I } \frac{[\mathbf{n}(I')]\cdot q^{|\mathbf{m}(I')| j}T^{(2)}_{\Omega \setminus I'}(q^j)}{j}s^j \biggr\}\\[0.7em]
&=\exp\biggl\{\sum_{I' \subseteq \Omega }\sum_{I \supseteq I'}\sum_{j \in E_I} \frac{[\mathbf{n}(I')]\cdot q^{ |\mathbf{m}(I')| j}T^{(2)}_{\Omega \setminus I'}(q^j)}{j}s^j \biggr\}\\[0.7em]
&=\exp\biggl\{\sum_{I' \subseteq \Omega }\sum_{j \in F_{I'}} \frac{[ \mathbf{n}(I')]\cdot q^{|\mathbf{m}(I')| j}T^{(2)}_{\Omega \setminus I'}(q^j)}{j}s^j \biggr\}.\\
\end{align*}
Here the third equality follows from $(\Omega \setminus I) \cup (I\setminus I')=\Omega \setminus I'$. Moreover since $F_I=[\mathbf{n}(I)]\cdot \N$, we calculate as
\begin{align*}
Z_q(s; \sigma_{n_1,m_1}, \ldots, \sigma_{n_r,m_r})
&=\exp\biggl\{\sum_{I' \subseteq \Omega}\sum_{k=1}^{\infty} \frac{ q^{|\mathbf{m}(I')| k  [\mathbf{n}(I')]}T^{(2)}_{\Omega \setminus I'}(q^{k[\mathbf{n}(I')]})}{k}s^{k[\mathbf{n}(I')]} \biggr\}\\[1.0em]
&=\exp\biggl\{\sum_{I' \subseteq \Omega }\sum_{k=1}^{\infty}\sum_{I'' \subseteq \Omega \setminus I' } \frac{ (-1)^{\# I''}q^{(|\mathbf{m}(I')|+|\mathbf{m}(I'')|)k|\mathbf{n}(I')|}}{k}s^{k [\mathbf{n}(I')]} \biggr\}\\[1.0em]
&=\exp\biggl\{\sum_{k=1}^{\infty}\sum_{I' \subseteq \Omega }\sum_{I'' \subseteq \Omega \setminus I' } \frac{ (-1)^{\# I''}q^{(|\mathbf{m}(I'\cup I'')|)k[\mathbf{n}(I')]}}{k}s^{k[\mathbf{n}(I')]} \biggr\}.\\
\end{align*}
Here the second equality follows from $(3.4)$.
Then, replacing $I' \cup I''$ to $J$, we have
\begin{align*}
\# I'' =\# J-\#I'.
\end{align*}
Furthermore, we replace $I'$ to $I$, hence we have 
\begin{align*}
Z_q(s; \sigma_{n_1,m_1}, \ldots, \sigma_{n_r,m_r})
&=\exp\biggl\{\sum_{k=1}^{\infty}\sum_{ I\subseteq J\subseteq \Omega}\frac{ (-1)^{\# J-\# I}q^{|\mathbf{m}(J)|k[\mathbf{n}(I)]}}{k}s^{k[\mathbf{n}(I)]} \biggr\}\\[1.0em]
&=\exp\biggl\{\sum_{ I \subseteq J \subseteq \Omega}(-1)^{\# J-\# I}\log(1-q^{|\mathbf{m}(J)|\cdot[\mathbf{n}(I)]}s^{[\mathbf{n}(I)]})^{-1}\biggr\}\\[1.0em]
\end{align*}
\begin{align*}
&=\prod_{ I \subseteq J \subseteq \Omega}\biggl\{ \frac{1}{1-q^{|\mathbf{m}(J)|\cdot[\mathbf{n}(I)]}s^{[\mathbf{n}(I)]}}\biggr\}^{(-1)^{\# J-\# I}}.\\
\end{align*}
On the other hand, by Theorem $3.1$, we have
\begin{align*}
\prod_{\emptyset \neq I \subseteq J \subseteq \Omega}\zeta(s, \sigma_{[\mathbf{n}(I)],|\mathbf{m}(J)|}; \beta_{[\mathbf{n}(I)],q})^{(-1)^{\# J-\# I}}
=\prod_{\emptyset \neq I \subseteq J \subseteq \Omega}\biggl\{ \frac{1-q^{|\mathbf{m}(J)|}s}{(1-s)(1-q^{|\mathbf{m}(J)|\cdot[\mathbf{n}(I)]}s^{[\mathbf{n}(I)]})} \biggr\}^{(-1)^{\# J-\# I}}.\\
\end{align*}
By Lemma $3.1.(2)$, we compute
\begin{align*}\hspace{-2.0em}
\prod_{\emptyset \neq I \subseteq J \subseteq \Omega}\biggl\{\frac{1}{1-s}\biggr\}^{(-1)^{\#J-\#I}}&=\exp\biggl\{ -\log(1-s)\sum_{\emptyset \neq I \subseteq J \subseteq \Omega}(-1)^{\#J-\#I} \biggr\}\\[1.0em]
&=\exp\biggl\{ -\log(1-s)\cdot 1\biggr\}\\[1.0em]
&=\frac{1}{1-s}.\\
\end{align*}
Furthermore,
\begin{align*}
\prod_{\emptyset \neq I \subseteq J \subseteq \Omega}(1-q^{|\mathbf{m}(J)|}s )^{(-1)^{\# J-\# I}}&=\prod_{\emptyset \neq J \subseteq \Omega}\prod_{\emptyset \neq I \subseteq J}(1-q^{|\mathbf{m}(J)|}s )^{(-1)^{\# J-\# I}}\\[1.0em]
&=\prod_{\emptyset \neq J \subseteq \Omega}\exp\biggl\{ \log (1-q^{|\mathbf{m}(J)|}s)\cdot \sum_{\emptyset \neq I \subseteq J}(-1)^{\# J-\# I} \biggr\}\\[1.0em]
&=\prod_{\emptyset \neq J \subseteq \Omega}\exp\biggl\{ \log (1-q^{|\mathbf{m}(J)|}s)\cdot (-1)^{\# J}\bigl\{\sum_{\emptyset \subseteq I \subseteq J}(-1)^{\# I}-(-1)^{\# \emptyset } \bigr\} \biggr\}\\[1.0em]
&=\prod_{\emptyset \neq J \subseteq \Omega}\exp\biggl\{ \log (1-q^{|\mathbf{m}(J)|}s)\cdot (-1)^{\# J}\bigl\{0-1 \bigr\} \biggr\}\\[1.0em]
&= \prod_{\emptyset \neq J \subseteq \Omega}\biggl\{ \frac{1}{1-q^{|\mathbf{m}(J)|}s }\biggr\}^{(-1)^{\# J}}.\\
\end{align*}
Here the forth equality follows from Lemma $3.1.(1)$ and $\# \emptyset =0$.
Hence we have
\begin{align*}
\prod_{\emptyset \neq I \subseteq J \subseteq \Omega}\zeta(s, \sigma_{[\mathbf{n}(I)],|\mathbf{m}(J)|}; \beta_{[\mathbf{n}(I)],q})^{(-1)^{\# J-\# I}}=\prod_{ I \subseteq J \subseteq \Omega}\biggl\{  \frac{1}{1-q^{|\mathbf{m}(J)|\cdot[\mathbf{n}(I)]}s^{[\mathbf{n}(I)]}}\biggr\}^{(-1)^{\# J-\# I}}.
\end{align*}
This completes the proof of Theorem $3.3$.
\end{proof}
By Theorem$1.2$ and Theorem $3.3$, we obtain the following corollary.
\begin{cor}
If the pairs $([\mathbf{n}(I)], |\mathbf{m}(I)|)\hspace{0.2em}$ are all coprime for any $I \subseteq \Omega$, 
and $n_1, \ldots, n_r$ are also coprime, then we have
\begin{align*}
\underset{s=1}{\Res}\hspace{0.2em} Z_q(s; \sigma_{n_1,m_1}, \ldots, \sigma_{n_r, m_r})=-\frac{1}{[n_1\cdots n_r]_q}\prod_{\emptyset \neq I \subseteq J \subseteq \Omega}\Delta_{T([\mathbf{n}(I)], |\mathbf{m}(J)|)}(q)^{(-1)^{\# J-\# I+1}}.
\end{align*}
\end{cor}
\begin{proof}
It is sufficient to compute the number of $[\mathbf{n}(I)]$.
\begin{align*}
\prod_{\emptyset \neq I \subseteq J \subseteq \Omega}\biggl\{\frac{1}{[[\mathbf{n}(I)]]_q}\biggr\}^{(-1)^{\# J-\# I}}
&=\prod_{\emptyset \neq I \subseteq \Omega}\exp \bigl\{ -\log([[\mathbf{n}(I)]]_q)\sum_{I \subseteq J \subseteq \Omega }(-1)^{\# J-\# I}\bigr\}\\[0.5em]
&=\prod_{\emptyset \neq I \subseteq \Omega}\exp \bigl\{ -\log([[\mathbf{n}(I)]]_q)\delta_{I, \Omega}\bigr\}\\[0.5em]
&=\frac{1}{[n_1\cdots n_r]_q}.
\end{align*}
Here the second equality follows from Lemma $3.1.(1)$. This proves Corollary $3.6$.
\end{proof}
\begin{ex}\normalfont
We give an example of the case $r = 2$.
We choose $n_1, n_2 \in \N_{\geq 2} , m_1, m_2 \in \Z$ such that $(n_1,n_2), (n_1, m_1), (n_2, m_2 )$ and $(n_1n_2, m_1+m_2)$ are all coprime. Then we compute 
\begin{align*}
\mathrm{tr}((\beta_{n_1,q}\sigma_{n_1, m_1})\otimes\beta_{n_2, q}(\sigma_{n_2, m_2}))^j)
&=\mathrm{tr}\beta_{n_1,q}(\sigma_{n_1,m_1}^j) \mathrm{tr}\beta_{n_2,q}(\sigma_{n_2,m_2}^j)\\[0.5em]
&=
\begin{cases}
(1-q^{m_1j})(1-q^{m_2j})
& j \in E_{\Set{\emptyset}},\\[0.5em]
(1-q^{m_1j}+n_1q^{m_1j})(1-q^{m_2j})
& j \in E_{\Set{1}},\\[0.5em]
(1-q^{m_1j})(1-q^{m_2j}+n_2q^{m_2j})
& j \in E_{\Set{2}},\\[0.5em]
(1-q^{m_1j}+n_1q^{m_1j})(1-q^{m_2j}+n_2q^{m_2j})
& j \in E_{\Set{1,2}}.
\end{cases}
\end{align*}
\begin{figure}[!h]
\begin{center}
\includegraphics[,width=6.0cm]{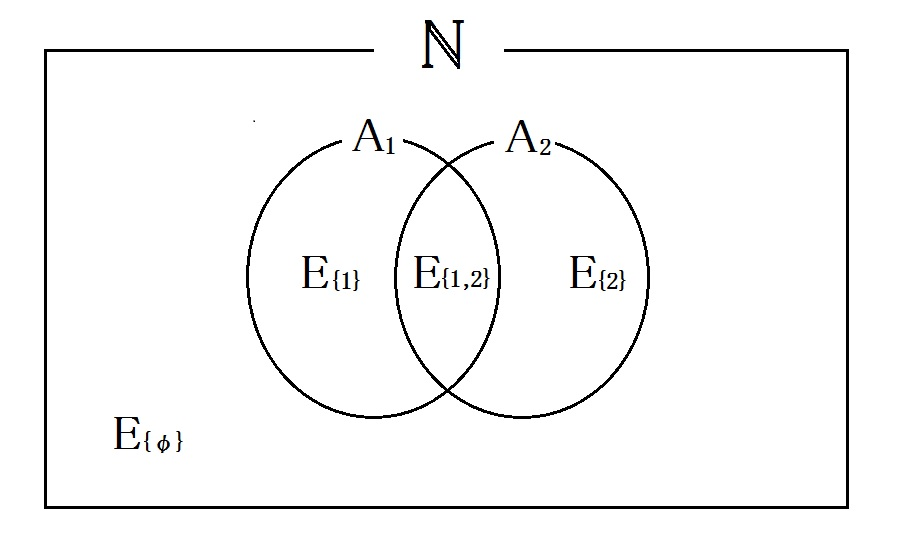} 
\end{center}
\end{figure}\\
Hence,
\begin{align*}
Z_q(s; \sigma_{n_1, m_1}, \sigma_{n_2, m_2})
&=\exp\biggl\{ \bigl\{\sum_{j \in E_{\Set{\emptyset}}}+\sum_{ j \in E_{\Set{1}}}+\sum_{ j \in E_{\Set{2}}}+\sum_{ j \in E_{\Set{1,2}}}\bigr\}\frac{(1-q^{m_1j})(1-q^{m_2j})}{j}s^j\\[0.5em]
&\hspace{1.0em}+\bigl\{\sum_{ j \in E_{\Set{1}}}+\sum_{ j \in E_{\Set{1,2}}}\bigr\}\frac{n_1q^{m_1j}(1-q^{m_2j})}{j}s^{j}+\bigl\{\sum_{ j \in E_{\Set{2}}}+\sum_{ j \in E_{\Set{1,2}}}\bigr\}\frac{n_2q^{m_2j}(1-q^{m_1j})}{j}s^j\\[0.5em]
&\hspace{1.0em}
+\sum_{ j \in E_{\Set{1,2}}}\frac{n_1n_2q^{(m_1+m_2)j}}{j}s^{j}\biggr\}\\[0.5em]
\end{align*}
\begin{align*}
&=\exp\biggl\{ \sum_{j \in F_{\Set{\emptyset}}}\frac{(1-q^{m_1j})(1-q^{m_2j})}{j}s^j
+\sum_{ j \in F_{\Set{1}}}\frac{n_1q^{m_1j}(1-q^{m_2j})}{j}s^{j}\\[0.5em]
&\hspace{7.0em}+\sum_{ j \in F_{\Set{2}}}\frac{n_2q^{m_2j}(1-q^{m_1j})}{j}s^j
+\sum_{ j \in F_{\Set{1,2}}}\frac{n_1n_2q^{(m_1+m_2)j}}{j}s^{j}\biggr\}\\[0.5em]
&=\exp\biggl\{ \sum_{j=1}^{\infty}\frac{(1-q^{m_1j})(1-q^{m_2j})}{j}s^j+\sum_{k=1}^{\infty}\frac{n_1q^{m_1kn_1}(1-q^{m_2kn_1})}{kn_1}s^{kn_1}\\[0.5em]
&\hspace{7.0em}+\sum_{k'=1}^{\infty}\frac{n_2q^{m_2k'n_2}(1-q^{m_1k'n_2})}{k'n_2}s^{k'n_2}+\sum_{l=1}^{\infty}\frac{n_1n_2q^{(m_1+m_2)ln_1n_2}}{ln_1n_2}s^{ln_1n_2}\biggr\}\\[0.5em]
&=\exp\biggl\{ \sum_{j=1}^{\infty}\frac{1-q^{m_1j}-q^{m_2j}+q^{(m_1+m_2)j}}{j}s^j+\sum_{k=1}^{\infty}\frac{q^{m_1kn_1}-q^{(m_1+m_2)kn_1}}{k}s^{kn_1}\\[0.5em]
&\hspace{7.0em}+\sum_{k'=1}^{\infty}\frac{q^{m_2k'n_2}-q^{(m_1+m_2)k'n_2}}{k'}s^{k'n_2}+\sum_{l=1}^{\infty}\frac{q^{(m_1+m_2)ln_1n_2}}{l}s^{ln_1n_2}\biggr\}\\[0.5em]
&=\frac{(1-q^{m_1}s)(1-q^{m_2}s)(1-q^{(m_1+m_2)n_1}s^{n_1})(1-q^{(m_1+m_2)n_2}s^{n_2})}{(1-s)(1-q^{m_1+m_2}s)(1-q^{m_1n_1}s^{n_1})(1-q^{m_2n_2}s^{n_2})(1-q^{(m_1+m_2)n_1n_2}s^{n_1n_2})}\\[1.5em]
&=\frac{\zeta(s, \sigma_{n_1, m_1}; \beta_{n_1, q})\zeta(s,\sigma_{n_2, m_2}; \beta_{n_2, q})\zeta(s, \sigma_{n_1n_2, m_1+m_2}; \beta_{n_1n_2, q})}{\zeta(s, \sigma_{n_1, m_1+m_2}; \beta_{n_1, q})\zeta(s, \sigma_{n_2, m_1+m_2}; \beta_{n_2, q})}.
\end{align*}
Then we have
\begin{align*}
\underset{s=1}{\Res}\hspace{0.2em} Z_q(s; \sigma_{n_1, m_1}, \sigma_{n_2, m_2})=-\frac{\Delta_{T(n_1, m_1+m_2)}(q)\Delta_{T(n_2, m_1+m_2)}(q)}{[n_1 n_2]_q \Delta_{T(n_1, m_1)}(q)\Delta_{T(n_2, m_2)}(q)\Delta_{T(n_1n_2, m_1+m_2)}(q)}.\\
\end{align*}
\end{ex}

\medskip
\begin{flushleft}
Kentaro Okamoto\\
Department of Mathematics\\
Kyushu University\\
Nishi-ku, Fukuoka 819-0395\\
Japan\\
(k-okamoto@math.kyushu-u.ac.jp)
\end{flushleft}

\end{document}